\documentclass[10pt]{amsart}
\usepackage[letterpaper, margin=1.5in]{geometry}
\usepackage{graphicx}
\usepackage{amsmath,amssymb,amsthm,mathtools}
\usepackage[all]{xy}
\usepackage[utf8]{inputenc}
\usepackage[hidelinks]{hyperref}
\usepackage{enumitem}

\newtheorem{theorem}{Theorem}[section]

\newtheorem*{conjecture}{Conjecture}
\newtheorem{remark}{Remark}[theorem]
\newtheorem{definition}{Definition}[section]
\newtheorem{lemma}[theorem]{Lemma}
\newtheorem{proposition}[theorem]{Proposition}
\newtheorem{question}[theorem]{Question}
\newtheorem{example}[theorem]{Example}
\newtheorem*{theorem*}{Theorem}
\newtheorem*{question*}{Question}
\begin{document}

\title{Closed points on cubic hypersurfaces}
\date{\today}

\author{Qixiao Ma}
\address{Department of Mathematics, Columbia University,}
\email{qxma10@math.columbia.edu}

\begin{abstract}
   We generalize some results of Coray on closed points on cubic hypersurfaces. We show certain symmetric products of cubic hypersurfaces are stably birational.
\end{abstract}

\maketitle

\tableofcontents
\section{Introduction}
Let $k$ be a field, let $X$ be a variety defined over $k$. By definition, a $k$-rational point on $X$ gives a zero-cycle of degree $1$ on $X$. Conversely, one may ask:
\begin{question*}If $X$ has a zero-cycle of degree $1$, does $X$ necessarily have a $k$-rational point?
\end{question*}
Here are some partial answers to this question:
\begin{itemize}[leftmargin=0.8cm]
\item Yes, if $X$ is a torsor of an abelian variety, see \cite{LT}.
\item Yes, if $X$ is a Brauer-Severi variety, i.e. $X_{\overline{k}}\cong\mathbb{P}^{\mathrm{dim}(X)}_{\overline{k}}$, see \cite{GS}.
\item No, if we only ask $X$ to be a geometrically rational variety, see \cite{CTC}.
\item Serre asked the question for $X$ being homogeneous space of algebraic groups, see \cite{Se}, \cite{tot}, \cite{bla} and \cite{Pari} for various results.
\end{itemize}

When $X$ is a cubic hypersurface, Cassels and Swinnerton-Dyer conjectured ``yes'':
\begin{conjecture}[CS, \cite{Coray}] Let $F(X_0,\dots,X_n)$ be a cubic form with coefficients in a field $k$. Suppose $F$ has a non-trivial solution in an algebraic extension $K/k$ of degree $d$ prime to $3$. Then $F$ also has a non-trivial solution in the ground field $k$.
\end{conjecture}

Coray worked over perfect fields, and reduced the conjecture for $n=3$ to smooth cubic surfaces with a closed point of degree $4$ or $10$, see \cite[7.1]{Coray} and \cite[8.2]{Coray}. We remove the perfectness assumption and study some related questions.

In section \ref{sec3}, after briefly review Coray's result, we show:
\begin{theorem*}
Let $k$ be a field, let $X$ be a regular cubic surface over $k$. Then the existence of a closed point $P$ on $X$ with $3\nmid\mathrm{deg}(P)$ implies that $X$ contains a closed point of degree $1$, $4$ or $10$.
\end{theorem*}
\noindent In section \ref{sec4}, following Coray's idea of taking intersection with rational normal curves, we show in some cases, one can still descend the degree of closed points on cubic 3-folds and 4-folds.

\noindent In section \ref{sec5}, we reinterpret the previous results in terms of rational map between symmetric products. In section \ref{sec6}, we study the rationality of the moduli spaces which parameterize rational normal curves passing through a collection of closed points. In section \ref{sec7}, we use the previous result to show stable birationality between symmetric products of cubic hypersurfaces:
\begin{theorem*}Let $k$ be a field, let $X$ be a smooth cubic 3-fold (4-fold), then $\mathrm{Sym}^7(X)$ and $\mathrm{Sym}^5(X)$ (resp. $\mathrm{Sym}^8(X)$ and $\mathrm{Sym}^7(X)$) are stably birational.
\end{theorem*}

\textbf{Acknowledgements.} I am very grateful to my advisor Aise Johan de Jong for his invaluable ideas, enlightening discussions and unceasing encouragement.
\section{Cubic surfaces}\label{sec3}
\subsection{Coray's results}
Coray used geometric arguments to descend the degree of closed points on cubic surfaces. We briefly review the argument for smooth cubic surfaces \cite[7.1]{Coray} and explain some difficulties to carry out Coray's argument over imperfect fields. Here is the statement:

\begin{theorem}[Coray]{\label{Th}}Let $k$ be a perfect field. Let $X\subset\mathbb{P}^3_k$ be a smooth cubic surface. If $X$ contains a closed point $P$ of degree prime to $3$, then $X$ contains a closed point $Q$ of degree $1,4$ or $10$.
\end{theorem}
When $k$ is a finite field, the theorem follows from Chevalley's theorem, see \cite[I.2.2]{Serre}. For infinite fields, the theorem follows from repeated application of the following lemma:
\begin{lemma} Let $k$ be an infinite perfect field. Let $X\subset\mathbb{P}^3_k$ be a smooth cubic surface. Let $d$ be a positive integer such that $d\neq 4,10$ and $3\nmid d$. If $X$ contains a closed point $P$ of degree $d$, 
then $X$ contains a point $Q$ of degree $d'<d$ and $(d',3)=1$.
\end{lemma}

\begin{proof}(Sketch) Carefully choose a low degree surface $F\subset\mathbb{P}^3_k$ that contains $P$. Denote the curve $X\cap F$ by $C$. If $C$ is smooth over $k$, we find a divisor $D$ on $C$ such that $g(C)<\mathrm{deg}(D)<d$ and $3\nmid\mathrm{deg}(D)$. Riemann-Roch theorem implies $\mathrm{dim}|D|>0$. A general effective divisor in $|D|$ contains a closed point of degree $d'<d$ and $(d',3)=1$. If $C$ is not smooth, we replace $C$ by $C'=(C_{\mathrm{red}})^\nu$: normalization of the reduced subscheme. We show $P$ remains a Cartier divisor on $C'$, then argue as the previous case, see \cite[7.1(B)]{Coray}.
\end{proof}

\begin{remark}
In \cite{Coray}, Coray's notation for varieties are subvarieties of $\mathbb{P}^n(\overline{k})$, and his arguments are carried out on $C_{\overline{k},\mathrm{red}}$.
When $k$ is perfect, all constructions descend to $k$. If $k$ is not perfect, and suppose $P$ is not a separable point, then it is not clear if $P_{\overline{k}}$ is still a subscheme of $C_{\overline{k},\mathrm{red}}$, and we cannot directly apply the arguments in \cite[7.1(B)]{Coray}.
\end{remark}

\begin{remark} We explain where the argument breaks down for $d=4,10$. When $d=3l(l-1)/2+1$, in the argument, one needs to take a degree $l$ surface $F$ passing through $P$, such that $F_{\overline{k}}$ is tangent to $X_{\overline{k}}$ at three points in the complement of $P_{\overline{k}}$. These impose $d+9$ conditions on the choice of $F$. The degree $l$ surfaces that do not contain $X$ form a ${l+3\choose 3}-{l\choose 3}=3l(l+1)/2-1$ dimensional family.
For $l=2,3$, we cannot always pick such $F$ as $3l(l+1)/2-1<d+9$.
\end{remark}

\subsection{Regular cubic surfaces}
We show Theorem \ref{Th} still holds even if $k$ is not perfect. More precisely, for regular cubic surfaces, we adopt a lifting argument and get Theorem \ref{Pr}.

\begin{lemma}\label{Lm}Let $k$ be a field. Let $R$ be an artinian $k$-algebra, such that $3\nmid \mathrm{dim}_k(R)$. Then $R$ contains a maximal ideal $\mathfrak{m}$, such that $3\nmid[\kappa(\mathfrak{m})\colon k]$.\end{lemma}
\begin{proof} Note that $R$ is a direct sum of Artinian local rings, we may pick a summand $(R_1,\mathfrak{m}_1)$ such that $3\nmid\mathrm{dim}_k(R_1)$. Note $R_1$ has a filtration of $R_1/\mathfrak{m}_1$-modules, so $[\kappa(\mathfrak{m}_1)\colon k]$ divides $\mathrm{dim}_k(R_1)$. Then $3\nmid\mathrm{dim}_k(R_1)$ implies $3\nmid[\kappa(\mathfrak{m}_1)\colon k]$.
\end{proof}

\begin{lemma}\label{hensel}
Let $(R,\mathfrak{m},\kappa)$ be a Henselian local ring, let $\mathcal{X}$ be a flat $R$-scheme, let $X$ be its fiber over $\mathfrak{m}$, then the image of specialization map $\mathcal{X}(R)\to X(\kappa)$ contains all smooth $\kappa$-points of $X$. \end{lemma}
\begin{proof}This follows from \cite[18.5.17]{EGA4} and openness of smooth loci \cite[01V9]{SP}.
\end{proof}

\begin{lemma}\label{keylem}
Let $X$ be a regular scheme over a field $k$, let $P\subset X$ be a closed point on $X$ of degree $d$, then the Hilbert scheme $\mathrm{Hilb}_{X/k}^{d}$ is smooth at $[P]$.
\end{lemma}
\begin{proof} By regularity of $\mathcal{O}_{X,P}$, the closed immersion $P\to X$ is a local complete intersection. Then $[P]$ corresponds to a smooth point of $\mathrm{Hilb}^d_{X/k}$, because local complete intersection subschemes represent smooth points on Hilbert schemes, see \cite[06D9]{SP}.
\end{proof}

\begin{theorem}\label{Pr}Let $k$ be a field. Let $X\subset\mathbb{P}^3_k$ be a regular cubic surface. If $X$ contains a closed point $P$ of degree $d$ prime to $3$, then $X$ contains a closed point $Q$ of degree $1,4$ or $10$.
\end{theorem}
\begin{proof} If $k$ has characteristic $0$, this follows from Theorem \ref{Th}. Suppose the characteristic of $k$ is $p$.
We may take a complete discrete valuation ring $R$, such that $p$ is a uniformizer and $R/p\cong k$, see \cite[Tag 0328]{SP}. Let's denote the fraction field of $R$ by $K$. We take any cubic surface $\mathcal{X}/R$ in $\mathbb{P}^3_R$, such that the special fiber is $X/k$.

Given a closed point $P$ on $X$ of degree $d$ prime to $3$. 
We apply Lemma \ref{hensel} and Lemma \ref{keylem}, the $k$-point $[P]$ on the special fiber $\mathrm{Hilb}_{X/k}^d$ lifts to a $R$-point on $\mathrm{Hilb}^d_{\mathcal{X}/R}$.

This $R$-point restricts to a $K$-point $[Q]$ on $\mathrm{Hilb}^d_{X_K/K}$, hence represents a closed subscheme $Z$ of degree $d$ on $X_K$. We may assume $Z$ is a closed point, for otherwise, by Lemma \ref{Lm}, we may take a closed point $Z'\subset Z$ of degree $d'<d$ and $3\nmid d'$, then continue with $Z'$.

If $d\neq 4, 10$, then Theorem \ref{Th} implies existence of a $K$-point of $\mathrm{Hilb}^e_{X_K/K}$, with $e<d$ and $3\nmid e$. Since the Hilbert scheme is proper
, by valuative criteria, this $K$-point extends to an $R$-point of $\mathrm{Hilb}^e_{\mathcal{X}/R}$. Specializing this $R$ point to a $k$-point of $\mathrm{Hilb}^e_{X/k}$, we get a subscheme of degree $e$ on $X$. By Lemma \ref{Lm}, we know $X$ contains a closed point whose degree is no greater than $e$ and prime to $3$. Repeating the argument, we descend $d$ to $1,4$ or $10$.
\end{proof}

\section{Cubic 3-folds and 4-folds}\label{sec4}
Coray's descent argument for cubic surfaces breaks down in higher dimension, as complete intersection curves have too high genus. In this section, we introduce a simple trick which brings down the degree of certain closed points on cubic 3-folds and 4-folds. Namely, we take residue intersection with rational normal curves. The main difficulty is to deal with the case when the geometrical points are not in linearly general position.

\subsection{Rational normal curves} Let $k$ be a field. We recall the fact that for any degree $n+3$ closed point in $\mathbb{P}^n_k$, whose geometric points are in general linear position, there exists a unique genus zero curve passing through it.

\begin{definition}Let $\overline{k}$ be an algebraically closed field. Given a finite set $\Gamma$ of closed points in $\mathbb{P}^r_{\overline{k}}$, we say $\Gamma$ is in linearly general position, if for every proper linear subspace $\Lambda\subset\mathbb{P}_{\overline{k}}^r$, we have $\deg(\Lambda\cap\Gamma)\leq1+\dim\Lambda$.
\end{definition}

Let $P$ be a separable closed point of degree $n+1$ in $\mathbb{P}_k^{n}$, such that its geometric points $\{P_1,\dots P_{n+1}\}$ are in linearly general position. We define the Cremona transformation at $P$, and use it to study the rational normal curves passing through $P$.

Let's assume $P$ lies in the affine piece $\{X_0\neq0\}\subset\mathbb{P}^n_k$.
Denote the coordinates of $P_i$ by $[a_{i,0}\colon\dots\colon a_{i,n}]$, with $a_{i,0}=1$. Let $\mathbf{a}_j$ be the row vector $(a_{i,0},\dots,a_{i,n})$. Let $\mathbf{x}=(X_0,\dots ,X_n)$. Let $A_j=(\mathbf{a}_0^T,\dots,\mathbf{a}_{j-1}^T,\mathbf{x}^T,\mathbf{a}^T_{j+1},\dots,\mathbf{a}_n^T)$
 be the $(n+1)\times(n+1)$ matrix. The points being in linear general position implies $\mathrm{det}(A_{j})\neq 0$. Let $f_i=\sum_{j=0}^na_{ij}\mathrm{det}(A_j)^{-1}$.
\begin{definition}
The Cremona involution at $P$ is defined as the rational map: $$\mathrm{Cr}_{P,k^{\mathrm{sep}}}\colon\mathbb{P}^n_{k^{\mathrm{sep}}}\dashrightarrow\mathbb{P}^n_{k^{\mathrm{sep}}},\ \ \ [x_0\colon\dots\colon x_n]\mapsto[f_0\colon\dots\colon f_n].$$
\end{definition}

\begin{lemma}
The rational map $\mathrm{Cr}_{P,k^{\mathrm{sep}}}$ descends to a rational map $\mathrm{Cr}_{P}$ defined over $k$.
\end{lemma}
\begin{proof} It suffices to show the inhomogeneous coordinates $f_i/f_0$ are $\mathrm{Gal}(k^{\mathrm{sep}}/k)$-invariant. Since $\mathrm{Gal}(k^{\mathrm{sep}}/k)$ acts on $\mathbf{a}_j$ by permutation, it suffices to show the rational function are invariant under the full permutation group $S_{n+1}=\langle(01),(01\dots n)\rangle$. This follows as $(01\dots n)$ map $f_i$ to $f_i$, and $(01)$ map $f_i$ to $-f_i$.
\end{proof}

\begin{lemma}\label{Rnc} The birational map $\mathrm{Cr}_P$ induces a bijection between smooth genus zero curves in $\mathbb{P}^n_k$ passing through $P$, and lines in $\mathbb{P}^n_k$ not meeting the union of all $n-2$ planes determined by $P$, i.e., not meeting $\bigcup_{{I\subset\{1,\dots,n+1\} ,\#|I|=n-1}}\mathrm{span}(P_i|i\in I)$.

\end{lemma}
\begin{proof} By Galois descent, it suffices to show the claim after base change to $k^{\mathrm{sep}}$. Since $\mathrm{PGL}_{n+1}$ action preserves degree, we may assume $\mathrm{Cr}_P$ has the form $[X_0\colon\dots\colon X_n]\mapsto[X_0^{-1}\colon\dots\colon X_n^{-1}]$, then check directly.
\end{proof}

\begin{proposition}\label{Rnc3} Let $P$ be a degree $n+3$ closed point in $\mathbb{P}^n_k$, geometrically in linearly general position, then there exists a unique smooth genus zero curve passing $P$.
\end{proposition}
\begin{proof}
After base change to $k^{\mathrm{sep}}$, we may choose $n+1$ geometric points $P'\subset (P)_{k^{\mathrm{sep}}}$ and perform the Cremona transformation at $P'$. By Lemma \ref{Rnc}, there is a unique genus zero curve passing through $P$. The uniqueness implies that the curve descends to $k$. \end{proof}

\subsection{Generalization and specialization}
We deal with the case when the geometric points are not in linear general position. We generalize the points to linearly general position, then specialize the rational normal curve back.

\begin{lemma}\label{gen} Let $n$ be a positive integer, let $K/k$ be a separable field extension of degree $d$. Then there exists a closed point $P\in\mathbb{A}_k^n$ with $\kappa(P)\cong K$ such that $P_{\overline{k}}\subset \mathbb{A}^n_{\overline{k}}$ consists of $d$ points in linearly general position.
\end{lemma}
\begin{proof} Since $K/k$ is separable, there exists an element $\alpha\in K$ such that $K=k[\alpha]$. Let $f_{\alpha}(t)=t^d+a_1t^{d-1}+\dots+a_d$ be the minimal polynomial of $\alpha$. Let $\{\alpha=\alpha_1,\alpha_2,\dots,\alpha_d\}$ be roots of $f_\alpha(t)$.

Consider the point $P\colon\mathrm{Spec}(K)\to\mathbb{A}^n_k$ given by $(1,\alpha,\dots,\alpha^{n-1})$. The geometric points lying over $P$ are $\{P_i=(1,\alpha_i,\alpha_i^2,\dots,\alpha_i^{n-1})\}_{i=1}^d\subset\mathbb{A}^n_{\overline{k}}$. These points are in linearly general position, by properties of the Vandermonde matrices.
\end{proof}

\begin{lemma}\label{Qing}Let $R$ be a discrete valuation ring. Let $S=\mathrm{Spec}(R)$, with generic point $\eta$ and closed point $s$. Let $X_\eta\subset \mathbb{P}^n_{\kappa(\eta)}$ be a geometrically rational curve. Let $\mathcal{X}\subset\mathbb{P}^n_S$ be the scheme theoretic closure of $X_\eta$. Then the irreducible components of the special fiber $X_{s}$ are geometrically rational curves.
\end{lemma}

\begin{proof}By L\"uroth's theorem, it suffices to show the irreducible components in $X_s$ are dominated by geometrically rational curves. Note that taking scheme theoretic closure commutes with flat base change \cite[0CMK]{SP}, 
we assume there is an isomorphism $f\colon X_\eta\cong\mathbb{P}^1_\eta$, which induces a rational map $f\colon\mathbb{P}^1_{S}{\dashrightarrow}\mathbb{P}^n_{S}$. By valuative criteria, the birational map has indeterminacy loci in codimension $2$. Thus there exists a subscheme $W\subset\mathbb{P}^1_{s}$, such that the blow-up $\pi\colon \mathrm{Bl}_W(\mathbb{P}^1_{S})\to\mathbb{P}^1_{S}$ resolves the indeterminacy of $f$, see \cite[II.7,17.3]{H}. By \cite[9.2.2]{Liu}, the morphism $\pi$ factors into a finite sequence of blow ups at closed points. Hence $X_s$ is dominated by a union of rational curves.
\end{proof}

\subsection{Intersection with rational normal curves}
\begin{theorem}\label{Cubic4}Let $k$ be a field. Let $X$ be a smooth cubic hypersurface in $\mathbb{P}_k^5$ (resp. $\mathbb{P}^4_k$), suppose $X$ has a closed point $x$ of degree $8$ (resp. $7$), then $X$ has a closed point of degree $1,2,4,5$ or $7$ (resp. $1,2,4$ or $5$).
\end{theorem}
\begin{proof}Let's deal with the case of cubic $4$-fold, the case of $3$-fold is the same. We assume $\kappa(x)/k$ is separable, otherwise we may reduce to $\mathrm{char}(k)=0$ case by the previous lifting trick.
Let $P_1,\dots,P_8$ be the geometric points lying over $x$.

If $\{P_i\}_{i=1}^8$ are in linearly general position, by Lemma \ref{Rnc}, there exists a rational normal curve $C$ passing through $x$.
If $C\nsubseteq X$, then $C\cap X$ is a zero-cycle of degree $15$. Since the $P_i$ are Galois conjugate, the intersection $C\cap X$ has to be reduced at $P_i$. By Lemma \ref{Lm}, the complement of $x$ in $C\cap X$ gives a required closed point. If $C\subseteq X$, we may directly take intersection of $C$ with a hyperplane in $\mathbb{P}^5_k$.

If $\{P_i\}_{i=1}^8$ are not in linearly general position, by Lemma \ref{gen}, we pick a degree $8$ closed point $x'$ in $\mathbb{P}^5_k$, geometrically in linearly general position. Let $x',x$ both lie in some affine chart in $\mathbb{A}^5_k$, with coordinates $x=(c_1,\dots,c_5)$, and $x'=(c_1',\dots, c_5')$, where $c_i,c_i'\in k^{\mathrm{sep}}$.

Consider the morphism of $k$-schemes $$\phi\colon\mathbb{A}^1_{k^{\mathrm{sep}}}\to\mathbb{A}^5_k\times\mathbb{A}^1_k\subset\mathbb{P}^5_k\times\mathbb{A}^1_k,$$ $$k^{\mathrm{sep}}[t]\leftarrow k[t_1,\dots t_5,s]$$ given by $(t_1,\dots,t_5,s)\mapsto(tc_1+(1-t)c_1',\dots,tc_5+(1-t)c_5',t)$. Let $Z=\mathrm{Im}(\phi)$. Since being in linearly general position is an open condition, take the generic fiber of $Z\to\mathbb{A}^1_k$, we get a zero cycle $x^*$ of degree $8$ in $\mathbb{P}^5_{k(t)}$, geometrically in linearly general position. By Proposition \ref{Rnc3}, there exists a unique rational normal curve $C'\subset\mathbb{P}^5_{k(t)}$ passing $x^*$.
Take the closure $\overline{C'}$ of $C'$ in $\mathbb{P}^5_{\mathbb{A}^1_k}$, the reduced subscheme $C_1=(\overline{C'}|_{t=1})_{\mathrm{red}}$ passes through $x$, with $\deg (C_1)\leq5$.

\begin{enumerate}[leftmargin=*]
\item  If $C_1\cap X$ is a zero-dimensional. The corresponding zero cycle has degree $9,12$ or $15$ ($\mathrm{deg}(C_1)=3,4,5$). The complement of $x$ in $C_1\cap X$ gives zero cycle of degree $1$ or $4$ or $7$ on $X$.
\item If some components of $C_1$ is contained in $X$.
Let $C_2$ be the closure of such components.
\begin{itemize}[leftmargin=0.1cm]
\item If $\mathrm{deg}(C_2)\neq 3$, we intersect $C_2$ with a general $k$-hyperplane.
\item
If $\mathrm{deg}(C_2)=3$, note by Lemma \ref{Qing}, its geometric irreducible components are rational curves.  Either $C_2$ contains a $k$-rational curve of degree $1$ or $C_2$ is a nodal curve consisting of $3$ Galois conjugate rational curves. The second case cannot happen, as Galois transitivity implies all the points $\{P_i\}_{i=1}^8$ lie in the smooth locus or are the nodes, which forces the number of points to be a multiple of $3$.
\end{itemize}
\end{enumerate}
\end{proof} 
\section{Rational map between symmetric products}\label{sec5}
We reinterpret Theorem \ref{Pr} in terms of rational maps between symmetric products of cubic surfaces.

\begin{proposition} Let $k$ be a field. Let $X\subset\mathbb{P}^3_k$ be a smooth cubic surface. For any positive integer $d$ such that $3\nmid d$, there exists a rational map $\mathrm{Sym}^d(X)\dashrightarrow\mathrm{Sym}^e(X)$ with $e=1,4\ \mathrm{or}\ 10.$
\end{proposition}
\begin{proof}
Let $K$ be the function field of $\mathrm{Sym}^d(X)$. The universal family $U\subset X\times\mathrm{Sym}^d(X)$ gives a separable degree $d$ point of $X_K$.
Since $3\nmid d$, by Theorem \ref{Pr}, there exists a $K$-point on $\mathrm{Hilb}^e(X_K)$ for $e=1,4$ or $10$. Composing with the Hilbert-Chow morphism $\pi\colon\mathrm{Hilb}^e_{X_{K}/K}\to\mathrm{Sym}^e(X_K)$, see \cite[1.5]{Nkjm}, we obtain a $K$-rational point on $\mathrm{Sym}^e(X_K)=\mathrm{Sym}^e(X)\times_kK$.
Composing with the first projection, we get a $K$-point of $\mathrm{Sym}^e(X)$. Note that $K$ is the function field of $\mathrm{Sym}^d(X)$, this gives a rational map $\mathrm{Sym}^d(X)\dashrightarrow\mathrm{Sym}^e(X)$.
\end{proof}

Conversely, the existence of rational map give information about rational points.

\begin{lemma}{\cite[3.6.11]{Poonen}}\label{LN} Let $X\dashrightarrow Y$ be a rational map between $k$-varieties, where $Y$ is proper. If $X$ has a smooth $k$-point, then $Y$ has a $k$-point.
\end{lemma}

\begin{proposition}\label{lemmapoint}
Let $X$ be a smooth $k$-surface. Given a rational map $f\colon\mathrm{Sym}^d(X)\dashrightarrow\mathrm{Sym}^e(X)$, then existence of a length $d$ subscheme on $X$ implies the existence of a $k$-point on $\mathrm{Sym}^e(X)$.
\end{proposition}
\begin{proof}
Consider the composition of $f$ with the Hilbert-Chow morphism $f\circ \pi\colon\mathrm{Hilb}^d_{X/k}\to\mathrm{Sym}^d(X)\dashrightarrow\mathrm{Sym}^e(X)$. Note that $\mathrm{Hilb}^d_{X/k}$ is smooth, $\mathrm{Sym}^e(X)$ is proper, the conclusion follows from Lemma \ref{LN}.
\end{proof}

We explain the relation between $k$-points of $\mathrm{Sym}^e(X)$ and zero-cycles of degree $e$ on $X$.

\begin{example}{\cite[4.1.2]{Kol}} Let $k=\mathbb{F}_2(s,t)$ be the purely transcendental extension of $\mathbb{F}_2$. Consider the second symmetric product of $\mathbb{A}^2_k$ and the universal family $\mathbb{A}_{k,(x_1,y_1)}^2\times\mathbb{A}_{k,(x_2,y_2)}^2\to \mathrm{Sym}^2(\mathbb{A}^2_k)$. The corresponding map on coordinate rings is $$k[x_1+x_2,x_1x_2,y_1+y_2,y_1y_2,x_1y_1+x_2y_2]\hookrightarrow k[x_1,x_2,y_1,y_2]$$ The fiber over the $k$-rational point $(0,s,0,t,0)$ has coordinate ring $k[x_1,x_2,y_1,y_2]/(x_1+x_2,x_1x_2-s,y_1+y_2,y_1y_2-t,x_1y_1+x_2y_2)=k[x,y]/(x^2-s,y^2-t)$. The fiber is a reduced zero-cycle of length $4$.
\end{example}

However, this is the worst possible case.
\begin{proposition}\label{cycle}Let $k$ be a field of characteristic $p$. Let $X$ be a smooth $d$-dimensional variety over $k$.
Let $n=m\cdot p^h$ be a positive integer, with $(m,p)=1$, (if $p=0$, we use convention $p^h=0^0=1$). Then the existence of a $k$-point on $\mathrm{Sym}^n(X)$ implies the existence of a zero-cycle of degree $m\cdot p^{d\cdot h}$ on $X$.
\end{proposition}
\begin{proof}By \cite[I.4.2.5]{rydh} and \cite[II.8.11]{rydh}, a $k$-point of $\mathrm{Sym}^n(X)$ corresponds to a ``quasi-integral effective zero-cycle'' $Z=\sum a_i[Z_i]$ on $X$ of degree $n$. By \cite[II.8.10]{rydh}, this means $Z$ is a $\mathbb{Q}$-coefficient zero-cycle, and for each $Z_i$, the denominators of $a_i$ is divides $p^{h_i-e_i}$, where $p^{e_i}=\mathrm{min}\{p^e|\kappa(Z_i)^{p^e}\subset k^{\mathrm{sep}}\}$ and $[\kappa(Z_i)\colon k]_{\mathrm{insep}}=p^{h_i}$. Let $f_i=-v_p(a_i)$, then $f_i\leq h_i-e_i$. Let $[\kappa(Z_i)\colon k]_{\mathrm{sep}}=m_i$.
Suppose $Z_0$ is the component where $h_i-f_i$ attains minimum, then $m_0\cdot p^{h_0-f_0}\leq n$ and $h_0-f_0\leq h$. It suffices to show $f_0\leq (h_0-f_0)(d-1)$ or $\frac{f_0}{h_0}\leq \frac{d-1}{d}$, because this implies $f_0\leq h(d-1)$, thus $\mathrm{deg}([Z_{0}])=m_0\cdot p^{h_0}\leq n\cdot p^{f_0}=m\cdot p^{h+f_0}\leq m\cdot p^{d\cdot h}$ and $Z_0$ gives the desired zero-cycle. Since $f_0\leq h_0-e_0$, it suffices to check $\frac{h_0-e_0}{h_0}\leq\frac{d-1}{d}$, namely $h_0\leq de_0$. Since $h_0$ and $e_0$ do not change if we base change to $k^{\mathrm{sep}}$, we may assume $k=k^{\mathrm{sep}}$ and $Z_0$ is a purely inseparable point.

Note we have closed immersion $i\colon \kappa(Z_0)\to X$.
Since $X$ is smooth, by \cite[Tag 054L]{SP}, there exists an open subset $U\subset X$ containing $Z_0$ and \'etale morphism $\phi\colon U\to \mathbb{A}^d_k$. By \cite[Tag 04XV]{SP}, we know $\phi\circ i$ is a closed immersion. Since we assumed $\kappa(Z_0)/k$ is purely inseparable, we may write $\kappa(Z_0)$ as $k[T_1,\dots,T_d]/(T_1^{p^{r_1}}-a_1,\dots,T_d^{p^{r_d}}-a_d)$. By definition of $e$, we know $r_i\leq e_0$. Thus $p^{h_0}=[\kappa(Z_0)\colon k]\leq (p^{e_0})^d$ and we conclude that $h_0\leq de_0$.
\end{proof}
\begin{remark} If $\mathrm{char}(k)=0$, by the proposition, a rational point on the $\mathrm{Sym}^d(X)$ implies the existence of a degree $d$ zero cycle on $X$.
\end{remark}
\section{The generic fibers}\label{sec6}
We study the rational map between symmetric products of cubic 3-folds and 4-folds, which are obtained by taking intersection with rational normal curves, see section \ref{sec4}. Such rational maps are dominant, and their generic fibers can be explicitly described. We study their rationality and apply the results to stably birationality of symmetric products of cubic hypersurfaces in the next section.

\subsection{The variety $F_{P.n}$}
Let $k$ be a field, let $i\colon P\hookrightarrow\mathbb{P}^n_k$ be a reduced subscheme of length $d$ in linearly general position. We define the variety $F_{P,n}$, which represents families of smooth genus zero curve passing through $P$.
Let's consider the functor $\mathcal{F}_{P,n}\colon\mathrm{Sch}/k\to\mathrm{Sets}$, which sends a $k$-scheme $T$ to the set of isomorphism classes of commutative diagrams of flat $T$-schemes
$$\xymatrix{
P\times_kT=\colon P_T\ar[rr]^{i\times 1_T}\ar[rd]& &\mathbb{P}^n_T\\
& C\ar[ru]^{i_C} &
}$$
where $C$ is a flat family of genus $0$ degree $n+1$ curves over $T$ and $i_C$ is closed immersion.

\begin{proposition}\label{fop} The functor $\mathcal{F}_{P,n}$ is represented by a $k$-scheme.
\end{proposition}
\begin{proof} Let $h(t)=t+n+1$, consider the open subscheme  $S=\mathrm{Hilb}^{h(t),\circ}_{\mathbb{P}^3_k/k}\subseteq \mathrm{Hilb}^{h(t)}_{\mathbb{P}^3_k/k}$ representing families of smooth degree $n+1$ genus zero curves. Let $\mathbb{P}^n_S\supset U\to S$ be the universal family. Consider the subscheme $P\times_kS=\colon P_S\subset \mathbb{P}^n_S$. Note the functor $\underline{\mathrm{Hom}}_{\mathbb{P}^n_S}(\mathcal{O}_U,\mathcal{O}_{P_S})\colon \mathrm{Sch}/S\to\mathrm{Sets}, T\mapsto \mathrm{Hom}_{\mathbb{P}^n_T}(\mathcal{O}_{U_T},\mathcal{O}_{P_T})$,  is represented by an $S$-scheme, see [Tag 08K6]. Let $Y$ be the component where the homomorphism is surjective (so $P_T$ is closed subscheme of $U_T$), then $Y$ represents the functor $\mathcal{F}_{P,n}$.
\end{proof}

\begin{definition}We denote the scheme representing $\mathcal{F}_{P,n}$ by $F_{P,n}$.
\end{definition}
\begin{remark}
Note that by Lemma \ref{Rnc}, if $\mathrm{deg}(P)=n+3$, then $F_{P,n}$ consists of a point. If $\mathrm{deg}(P)>n+3$, then $F_{P,n}$ is empty.
\end{remark}

\subsection{Rationality of $F_{P,n}$}
We study the rationality of $F_{P,n}$ when $\mathrm{deg}(P)\leq n+2$.

\begin{proposition}\label{general} The variety $F_{P,n}$ is stably rational if $2\nmid n$ or $2\nmid \mathrm{deg}(P)$.
\end{proposition}
\begin{proof} Let $m=n+3-\mathrm{deg}(P)$. Consider the rational map $f\colon \mathrm{Sym}^{m}(\mathbb{P}^n_k)\dashrightarrow F_{P,n}$, associating a general degree $m$ subscheme $Q$ to the unique genus zero curve passing through $P,Q$. Let $C\to F_{P,n}$ be the universal family. Let $K$ be the function field of $F_{P,n}$. Note that geometrically, the curve $C_K$ is a rational normal curve in $\mathbb{P}^n_K$. If $n$ or $\mathrm{deg}(P)$ is odd, then $C_K$ is a conic with a closed point of odd degree, thus $C_K\cong\mathbb{P}^1_K$. The generic fiber of $f$ is $W=\mathrm{Sym}^m(C_K)\cong\mathbb{P}_K^m$. Thus
$F_{P,n}\times \mathbb{P}^m_K \overset{\mathrm{birat.}}{\sim}\mathrm{Sym}^m(\mathbb{P}_k^n)$. We know $\mathrm{Sym}^m(\mathbb{P}^n_k)\overset{\mathrm{birat.}}{\sim}\mathbb{P}^{mn}_k$, see \cite{Mut}, thus $F_{P,n}$ is stably rational.
\end{proof}

\begin{proposition}\label{n+1}
The variety $F_{P,n}$ is rational if $\deg(P)=n+1$.
\end{proposition}
\begin{proof}
By Lemma \ref{Rnc}, the smooth genus zero curves of degree $n+1$ through the $n+1$ points corresponds to lines in $\mathbb{P}^n_k$ avoiding some codimension $2$ subset, thus $F_{P,n}$ is rational. \end{proof}

\begin{proposition}\label{n+2odd} The variety $F_{P,n}$ is rational if $\deg(P)=n+2$ and $n$ is odd.
\end{proposition}

We will give the rational parametrization explicitly. Let's do some preparation: Let $d$ be a positive integer, let $K/k$ be a separable extension of degree $2d+1$. Let $\theta$ be a primitive element. Let $P=\mathrm{Spec}(K)$. Let $i\colon P\to\mathbb{P}_k^{2d-1}$ be a fixed closed immersion.

\begin{lemma}\label{lemma1}
View $K$ as a vector space over $k$. Let $W\subset K$ be the subspace spanned by $\{1,\theta,\dots,\theta^{d}\}$. Let $b_1,b_2\in K$ be general elements, then up to $k^*$-scalar, there exists a unique element $\lambda\in K^*$ such that $\lambda b_1,\lambda b_2\in W$.\end{lemma}
\begin{proof}
Write $\lambda\in K$ as $\lambda=\sum_{i=0}^{2d} l_i\theta^i$, with $l_i\in k$. Since multiplication by $\lambda$ is a $k$-linear map on $K$, we may write $\lambda b_1=\sum_{i=0}^{2d} f_{1,i}\theta^i$, $\lambda b_2=\sum_{i=0}^{2d}f_{2,i}\theta^i$, where $f_{j,i}$ are $k$-linear forms in $l_i$. The conditions $\lambda b_1,\lambda b_2\in W$ impose $2d$ conditions on $l_i$: $f_{1,s}=f_{2,s}=0$, $s=d+1,\dots 2d$. We claim the solution to these $2d$ equations is a $1$-dimensional subspace for general $b_1,b_2$. We know dimension of the subspace of solutions is at least $1$, since there are $2d+1$ variables $\{l_i\}$ and $2d$ linear conditions. Since dimension of kernel is upper semi-continuous, it suffices to show there exists $b_1',b_2'$ such that the solution is $1$ dimensional. We choose $b_1'=1,b_2'=\theta^d$, then $\lambda \cdot 1\in W$ forces $l_s=0$ for $s=d+1,\dots,2d$ and $\lambda \cdot\theta^{d}\in W$ forces $l_s=0$ for $s=1,\dots,d$. The solution space is $\{(l_i)|l_0\in k, l_i=0,i>0\}$, which is $1$-dimensional.
\end{proof}

\begin{lemma}\label{lemma3}Let $C$ be a smooth geometrically irreducible genus zero curve. Let $j\colon P\to C$ be a general closed immersion. Then there exists a unique isomorphism $a\colon C\to\mathbb{P}^1_k$ such that $a\circ j$ is given by $[b_o\colon b_1]$, with $b_0,b_1\in K$ such that $b_0=\theta^{d}+\sum_{i=0}^{d-2}\alpha_{i}\theta^{i}$ and $b_1=\theta^{d-1}+\sum_{i=0}^{d-2}\beta_{i}\theta^{i}$. The choice for $\alpha_i,\beta_i\in k$ are unique.
\end{lemma}
\begin{proof}
Note $C$ is a conic with odd degree point $P$, so $C$ is a rational curve. Choose an isomorphism $C\to\mathbb{P}_k^1$. Let $a\circ j\colon P\to\mathbb{P}^1_k$ be given by coordinates $b_0',b_1'$. Apply Lemma \ref{lemma1}, we may assume $b_i'\in W$. Let's write $b'_0=\sum_{i=0}^d\alpha_i'\theta^i,b_1'=\sum_{i=0}^d\beta_i'\theta^i$. For general closed immersion $j$, we may assume the coefficient of $\alpha_d'\neq 0$ and $\alpha_d'\beta_{d-1}'-\alpha_{d-1}'\beta_d'\neq 0$. Then we may reduce $b_0',b_1'$ to the required form, using the row echelon reduction
$$\begin{bmatrix}
\alpha_d' & \alpha_{d-1}' &\alpha_{d-2}'&\dots &\alpha_0'\\
\beta_d' & \beta_{d-1}' &\beta_{d-2}'&\dots &\beta_0'\\
\end{bmatrix}\mapsto
\begin{bmatrix}
1 & 0 & \alpha_{d-2} & \dots &\alpha_0\\
0 & 1 & \beta_{d-2} & \dots &\beta_0\\
\end{bmatrix}$$The choice of $\alpha_i,\beta_i$ are unique because the choice of $b_0',b_1'$ are unique up to $k^*$-scalar.\end{proof}
\begin{remark} Conversely, for any $\alpha_i,\beta_i\in k$, we can define an embedding $j\colon P\to \mathbb{P}^1_k$. We say $\alpha,\beta$ is the associated $(2d-2)$-tuple to $j$ and $j$ the associated immersion of the $(2d-2)$-tuple. They uniquely determines each other.
\end{remark}

\begin{lemma}\label{lemma2}
Let $V_1,V_2\subset K$ codimension $1$ subspaces. Then there exist a unique element $\lambda\in K^*$ (up to $k^*$-scalar) such that $\lambda V=V'$.
\end{lemma}
\begin{proof}Since $K/k$ is separable, the trace pairing $(x,y)\mapsto \mathrm{tr}(xy)$ is a non-degenerate $k$-bilinear pairing. Then $\lambda V=V'\Leftrightarrow (\lambda V)^\perp=V'^\perp\Leftrightarrow{\lambda}^{-1}(V^\perp)=V'^{\perp}$.
Let $v,v'$ be generator of $V^\perp, V'^{\perp}$. Then up to a $k^*$-scalar, we have $\lambda=v/v'$.
\end{proof}

\begin{lemma}\label{lemma4}For a general $(2d-2)$-tuple $(\alpha_i,\beta_i)_{i=0}^{d-2}$, there exists a unique closed immersion $f_{(\alpha,\beta)}\colon \mathbb{P}^1_k\to\mathbb{P}^{2d-1}_k$ such that $i=f_{(\alpha,\beta)}\circ j$.\end{lemma}
\begin{proof}
The immersion $i$ determines a $2d$-tuple $(p_0,\dots,p_{2d-1})\in K^{2d}-\{0\}$, unique up to $K^*$-scaling. We use notations as in Lemma \ref{lemma3}. For general $(\alpha_i,\beta_i)_{i=0}^{2d}$. The vectors $\{b_0^{2d-1-j}b_1^j\}_{j=0}^{2d-1}$ are linearly independent. This is shown by semi-continuity of rank, and the fact that $\{b_0^{2d-1-j}b_1^j\}_{j=0}^{2d-1}$ are $k$-linearly independent if we take $\alpha_i,\beta_i$ all equal to $0$.

By Lemma \ref{lemma2}, we may fix $(p_0,\dots,p_{2d-1})$, unique up to $k^*$-scaling, such that they span the same subspace generated by $\{b_0^{2d-1-j}b_1^j\}_{j=0}^{2d-1}$. We may write $p_t=\sum_{i=0}^{2d-1} a_{t,i}b_0^ib_1^{2d-1-i}$. Let $\{X_t\}_{t=0}^{2d+1}$ be coordinates of $\mathbb{P}^{2d+1}_k$, let $S,T$ be coordinates of $\mathbb{P}^1_k$. Then
$f_{\alpha,\beta}^*X_t=\sum_{i=0}^{2d+1} a_{t,i}S^iT^{2d+1-i}$ define the unique required morphism.
\end{proof}
\begin{proof}[Proof of Proposition \ref{n+2odd}] Let $d=\frac{n+1}{2}$, by Lemma \ref{lemma3} and \ref{lemma4}, there exists a birational map $$\mathbb{A}^{2d-2}_{\alpha,\beta}\to F_{P,n},\ \ \ (\alpha,\beta)\mapsto f_{\alpha,\beta}.$$\end{proof}

\begin{remark} When $\mathrm{char}(k)\neq2$, Proposition \ref{n+2odd} can also be recovered from \cite[1.2]{Flor}, using Kapranov's identification of $(F_{P,n})_{\overline{k}}$ with $(\overline{M}_{0,n+2})_{\overline{k}}$, see \cite{Kap}.
\end{remark}

When both $n$ and $\mathrm{deg}(P)$ are even, we know $F_{P,n}$ is rational for $n=2,\mathrm{deg}(P)=4$, because the plane conics are parameterized by their equations. We may ask:
\begin{question} Is $F_{P,n}$ is rational when $n$ and $\mathrm{deg}(P)$ are both even?
\end{question} \section{Stable birationality}\label{sec7}
Let $k$ be a field with $\mathrm{char}(k)=0$. Let $P$ be a closed subscheme of $\mathbb{P}^n_k$, which geometrically consists of $n+3$ points in linearly general position. Denote the unique rational normal curve passing through $P$ by $C_P$.

\subsection{Cubic surface}
Let $X\subset\mathbb{P}^3_k$ be a smooth cubic surface.
\begin{proposition}The varieties $\mathrm{Sym}^6(X)$ and $\mathrm{Sym}^3(X)$ are stably birational.
\end{proposition}
\begin{proof}
Fix a rational plane $H$ in $\mathbb{P}^3_k$. There exist rational maps
\begin{align*}
\phi_1\colon &\mathrm{Sym}^6(X)\dashrightarrow\mathrm{Sym}^3(H), \ \ P\mapsto C_P\cap H,\\
\phi_2\colon &\mathrm{Sym}^6(X)\dashrightarrow\mathrm{Sym}^3(X), \ \ P\mapsto (C_P\cap X)\backslash P.
\end{align*}
Then we have the rational map $(\phi_1,\phi_2)\colon \mathrm{Sym}^6(X)\dashrightarrow\mathrm{Sym}^3(X)\times \mathrm{Sym}^3(H)$. This map is birational with inverse map $$(\phi_1,\phi_2)^{-1}\colon \mathrm{Sym}^3(X)\times \mathrm{Sym}^3(H)\dashrightarrow\mathrm{Sym}^6(X), \ \ (P_1,P_2)\mapsto (C_{\{P_1\cup P_2\}}\cap X)\backslash P_1.$$
Since $\mathrm{Sym}^3(H)$ is rational, see \cite{Mut}, we know $\mathrm{Sym}^6(X)$ and $\mathrm{Sym}^3(X)$ are stably birational.
\end{proof}

\begin{remark}It is not clear if $\mathrm{Sym}^3(X)$ is rational, although we know it is geometrically rational and it contains rational points parameterized by $\mathrm{Gr}(2,4)$.
\end{remark}
\subsection{Cubic 3-fold and 4-fold}
\begin{proposition}Let $n=3$ or $4$. Let $X\subset\mathbb{P}^{n+1}_k$ be a smooth cubic $n$-fold. The variety $\mathrm{Sym}^{n+4}(X)$ is birational to $\mathrm{Sym}^{2n-1}(X)\times\mathbb{P}_k^{15-3n}$.
\end{proposition}
\begin{proof} Consider the rational map $$f_n\colon\mathrm{Sym}^{n+4}(X)\dashrightarrow\mathrm{Sym}^{2n-1}(X),\ \ P\mapsto (C_P\cap X)\backslash P.$$ Note that rational normal curves in $\mathbb{P}^{n+1}_k$ are determined by $n+4$ points. Since $2n-1<n+4$, we know the map $f_n$ is dominant. If suffices to show the generic fiber of $f_n$ is rational.

Let $K$ be the function field of $\mathrm{Sym}^{2n-1}(X)$. Let $R\subset X_K\subset\mathbb{P}^{n+1}_K$ be the universal degree $2n-1$ zero-cycle. The generic fiber of $f_n$ is $F_{R,2n-1}$. The rationality of $F_{R,5}$ follows from Proposition \ref{n+1}, and the rationality of $F_{R,7}$ follows from Proposition \ref{n+2odd}.
\end{proof}
\begin{remark}If we only care about stable birationality, it suffices to note that $F_{R,2n-1}$ is stably rational by Proposition \ref{general}.
\end{remark}
\begin{remark}Let $X\subset\mathbb{P}^n_k$ be a smooth hypersurface of degree $m$. Let $l\in[1,\frac{m(n+3)}{2})$ be an integer. The same method shows $\mathrm{Sym}^l(X)$ and $\mathrm{Sym}^{m(n+3)-l}(X)$ are stably birational if either $l$ or $n$ is odd.
\end{remark}

\bibliographystyle{alpha}
\bibliography{references}

\end{document}